\documentclass[11pt,reqno]{amsart}

\usepackage[utf8]{inputenc}
\usepackage[T1]{fontenc}
\usepackage{lmodern}
\usepackage{microtype}
\usepackage{amsmath,amssymb,amsthm,mathtools}
\usepackage{aliascnt}
\usepackage[colorlinks=true,linkcolor=blue,citecolor=blue,urlcolor=blue]{hyperref}
\usepackage[capitalize,nameinlink]{cleveref}
\theoremstyle{plain}
\newtheorem{theorem}{Theorem}[section]
\newaliascnt{proposition}{theorem}
\newtheorem{proposition}[proposition]{Proposition}
\aliascntresetthe{proposition}
\newaliascnt{lemma}{theorem}
\newtheorem{lemma}[lemma]{Lemma}
\aliascntresetthe{lemma}
\newaliascnt{corollary}{theorem}
\newtheorem{corollary}[corollary]{Corollary}
\aliascntresetthe{corollary}
\theoremstyle{definition}
\newaliascnt{definition}{theorem}
\newtheorem{definition}[definition]{Definition}
\aliascntresetthe{definition}
\theoremstyle{remark}
\newaliascnt{remark}{theorem}
\newtheorem{remark}[remark]{Remark}
\aliascntresetthe{remark}
\numberwithin{equation}{section}

\crefname{proposition}{proposition}{propositions}
\Crefname{proposition}{Proposition}{Propositions}
\crefname{lemma}{lemma}{lemmas}
\Crefname{lemma}{Lemma}{Lemmas}
\crefname{corollary}{corollary}{corollaries}
\Crefname{corollary}{Corollary}{Corollaries}
\crefname{definition}{definition}{definitions}
\Crefname{definition}{Definition}{Definitions}
\crefname{remark}{remark}{remarks}
\Crefname{remark}{Remark}{Remarks}

\DeclareMathOperator{\Spa}{Spa}
\newcommand{\cO}{\mathcal O}
\newcommand{\N}{\mathbf N}
\newcommand{\Cp}{C_p}
\newcommand{\whotimes}{\mathbin{\widehat\otimes}}

\title[Finite quotients of sousperfectoid spaces]{Finite quotients of sousperfectoid adic spaces need not be adic}
\author[Xiao-Ming Fu]{Xiao-Ming Fu}
\address{School of Mathematical Sciences, University of Science and Technology of China, Hefei, Anhui 230026, P.R.\ China}
\email{fuxm@ustc.edu.cn}
\author[Tianyang Sun]{Tianyang Sun}
\address{School of Mathematical Sciences, University of Science and Technology of China, Hefei, Anhui 230026, P.R.\ China}
\email{tysun@mail.ustc.edu.cn}
\date{August 2026}

\begin{document}

\begin{abstract}
For every prime $p$ and every perfectoid field $K$ of characteristic $p$, we
construct a geometrically normal sousperfectoid affinoid adic space with a
$C_p$-action whose quotient in the category of $v$-ringed spaces is not adic.
The source is stably uniform and sheafy, while a cofinal sequence of rational
neighborhoods at a fixed point acquires new degree-one invariant sections that
do not descend by completed rational localization.  Consequently, finite
quotient stability for affinoid perfectoid spaces does not extend to
sousperfectoid spaces.
\end{abstract}

\maketitle

\section{Introduction}\label{sec:introduction}

Hansen asked when the quotient of an adic space by a finite group is again
adic~\cite[Introduction]{HansenQuotients}.  He proved that the quotient is adic
when the space admits a covering by open affinoids stable under the group
action and the group order is invertible
~\cite[Theorem~1.1]{HansenQuotients}.  For an affinoid source, Hansen further
showed that, under the same invertibility hypothesis, the invariant Huber pair
is sheafy and represents the quotient, while suggesting that invertibility
might be unnecessary
~\cite[Theorem~1.2 and the paragraph following it]{HansenQuotients}.  For a
perfectoid space admitting such a covering by affinoid perfectoid subspaces,
he proved that the quotient is perfectoid if either the group order is
invertible or the action is linear over a perfectoid base field
~\cite[Theorem~1.4]{HansenQuotients}.  Hansen and Johansson established a
global perfectoid result under analytic separation and an affinoid-neighborhood
condition on the closures of rank-one points
~\cite[Theorem~5.8]{HansenJohansson}.  Zavyalov treated spaces locally
topologically of finite type over a locally strongly noetherian analytic base,
subject to an affinoid-neighborhood condition on finite orbits
~\cite[Theorem~1.3.3]{ZavyalovQuotients}.  Outside these regimes the coarse
quotient need not be adic~\cite[Section~5.1]{HansenJohansson}.  We give an
explicit counterexample with a geometrically normal, stably uniform, and
sheafy source.  In particular, our example shows that the invertibility
hypothesis in Hansen's affinoid quotient theorem cannot be removed for
arbitrary sheafy Tate--Huber pairs, even when the source is geometrically
normal, stably uniform, and sousperfectoid.

We use Hansen's quotient in the category of $v$-ringed spaces, whose underlying
space is the finite topological quotient and whose structure sheaf is formed
from invariant sections~\cite[Definition~2.2 and Proposition~2.3]
{HansenQuotients}.

\begin{theorem}\label{thm:uniform-main}
Let $p$ be a prime number and let $K$ be a perfectoid field of characteristic
$p$.  There exist a complete sousperfectoid Tate--Huber pair $(R,R^\circ)$
over $K$ and a continuous action of $\Cp$ on $X=\Spa(R,R^\circ)$ such that
$(R,R^\circ)$ is stably uniform and sheafy, but the quotient $X/\Cp$ in the
category of $v$-ringed spaces is not an adic space.  More precisely, the image
of a $\Cp$-fixed $K$-point has no affinoid neighborhood.  Moreover,
$R\whotimes_KL$ is a normal integral domain for every complete extension
$L/K$.
\end{theorem}

The last assertion is geometric normality.  It removes the nilpotent and
nonuniform features of the elementary square-zero model that first suggests
the localization obstruction.

\begin{corollary}[The perfectoid--sousperfectoid boundary]
\label{cor:no-unrestricted-quotient}
Finite-quotient stability for affinoid perfectoid spaces does not extend to
affinoid sousperfectoid spaces, even for a cyclic group of prime order.
Stable uniformity and sheafiness of the source do not suffice to make the
quotient adic.  In particular, the ring $R$ in \Cref{thm:uniform-main} is
sousperfectoid but not perfectoid.
\end{corollary}

Here is the construction and the obstruction.  Choose
$\pi\in K$ with $0<|\pi|<1$, let $J$ be the positive-degree ideal in a
completed Tate algebra on countably many variables, and form the pinched ring
\[
 R=K\langle T,s\rangle+sJ.
\]
The pinching has two essential effects.  It realizes $R$ as a module retract
of a perfectoid completed monoid algebra, and it makes the bidiscs
\[
 |T|\leq|\pi|^\ell,\qquad |s|\leq|\pi|^{2\ell}
\]
cofinal at a distinguished point.  A triangular $\Cp$-action preserves the
pinching.  On the $\ell$th bidisc, its degree-one fixed vectors contain
precisely the recurrence blocks indexed by $j<\ell$.  Passing to the next
bidisc creates the new invariant $s u_{\ell,\ell+1}$.  A continuous
monomial-coordinate functional vanishes on the preceding completed
localization image and takes value one on this new section.  The persistent
failure along a cofinal neighborhood system rules out every affinoid chart at
the image point.

The source still has countably many variables and is not topologically of
finite type.  Thus the example does not conflict with the finite-type
quotient theorem.  We make no priority claim beyond giving an explicit
construction.

In \cref{sec:persistent-obstruction} we isolate the local principle that turns
the comparison failure into non-adicity.  The source and its cyclic action
are constructed in \cref{sec:uniform-construction}.  The cofinal bidiscs,
local invariant calculation, and proof of \Cref{thm:uniform-main} appear in
\cref{sec:uniform-quotient}.  A final proposition records a purely
inseparable feature of the particular missing sections without asserting
that perfection makes the full comparison map surjective.

\section{Persistent rational-localization obstructions}
\label{sec:persistent-obstruction}

Let $(A,A^+)$ be a complete sheafy Tate--Huber pair, let a finite group $G$
act continuously, and let
\[
 q\colon Z=\Spa(A,A^+)\longrightarrow Y=Z/G
\]
be Hansen's quotient in the category of $v$-ringed spaces.  Its sheaf satisfies
\[
 \cO_Y(U)=\cO_Z(q^{-1}(U))^G,
\]
and restriction of valuations identifies the underlying topological quotient
with the adic spectrum of the global invariant pair
~\cite[Theorem~1.2]{HansenQuotients}.  The topological identification alone
does not imply that $Y$ is adic: invariant sections must also satisfy the
rational-localization rule.

Call an open subset of a $v$-ringed space a \emph{globally defined rational
locus} if it is cut out by finitely many global sections through rational
inequalities whose restrictions define rational subsets in every affinoid
open under consideration.

\begin{proposition}[Persistent obstruction]\label{prop:persistent-obstruction}
Let $Y$ be a $v$-ringed space and let $y\in Y$.  Suppose that
\[
 V_0\supset V_1\supset V_2\supset\cdots
\]
is a cofinal sequence of open neighborhoods of $y$, each $V_n$ is a globally
defined rational locus, and $V_{n+1}$ is intrinsically a rational locus in
$V_n$ defined by finitely many inequalities.  If the canonical completed
rational-localization map
\[
 \cO_Y(V_n)\left\langle
   \frac{f_{n,1}}{s_{n,1}},\ldots,\frac{f_{n,r_n}}{s_{n,r_n}}
 \right\rangle
 \longrightarrow \cO_Y(V_{n+1})
\]
fails to be an isomorphism for every $n$, then $y$ has no affinoid
neighborhood in $Y$.
\end{proposition}

\begin{proof}
If $U$ were an affinoid neighborhood of $y$, cofinality would give an $n$
with $V_n\subset U$.  The global rational description would make $V_n$ an
affinoid rational subdomain of $U$, and the intrinsic description would make
$V_{n+1}$ an affinoid rational subdomain of $V_n$.  The rational-localization
property of affinoid adic spaces~\cite[Proposition~1.3]{HuberAdic} would make
the displayed map an isomorphism, a contradiction.
\end{proof}

\section{The pinched sousperfectoid source}
\label{sec:uniform-construction}

We now give the construction used in \Cref{thm:uniform-main}.

\begin{lemma}[Normality of the infinite polydisc]
\label{lem:infinite-tate-normal}
Let $L$ be a complete nonarchimedean field and let $\mathcal U$ be a
countable set of variables.  Then the completed infinite-variable Tate
algebra $L\langle\mathcal U\rangle$ is a completely integrally closed
domain.  In particular, it is normal.
\end{lemma}

\begin{proof}
Write $S=L\langle\mathcal U\rangle$ and enumerate $\mathcal U$.  The Gauss
norm on $S$ is multiplicative.  One way to see this is to order the free
abelian group on $\mathcal U$ compatibly with addition.  For each nonzero
$c_0$-series, only finitely many coefficients have maximal norm.  The least
maximal-norm exponent in each of two series gives a unique maximal-norm
contribution to their product.  In particular, the Gauss norm extends to a
multiplicative norm on $\operatorname{Frac}(S)$.

Let $q=a/b\in\operatorname{Frac}(S)$ be almost integral over $S$, with
$b\neq0$.  Thus there is a nonzero $c\in S$ such that
$cq^m\in S$ for every $m\geq0$.  Choose a finite set
$F_0\subset\mathcal U$ containing the variables occurring in a
maximal-norm monomial of $b$ and in a nonzero monomial of $c$.  Let $F_n$ be
the union of $F_0$ with the first $n$ variables, and let
\[
 a_n,\ b_n,\ c_n\in S_n:=L\langle F_n\rangle
\]
be obtained by setting the variables outside $F_n$ equal to zero.  Then
$\lVert b_n\rVert=\lVert b\rVert$, while $b_n,c_n\neq0$.  The projected
relations
\[
 c_nq_n^m\in S_n\qquad(m\geq0),\qquad q_n:=a_n/b_n,
\]
show that $q_n$ is almost integral over the finite-variable Tate algebra
$S_n$.  This algebra is noetherian and factorial, hence normal
~\cite[Theorem~5.2.6/1]{BGR}.  Since
$S_n[q_n]\subset c_n^{-1}S_n$ and the latter is a finite $S_n$-module,
noetherianity makes $S_n[q_n]$ finite over $S_n$.  Thus $q_n$ is integral
over $S_n$, and normality gives
\[
 q_n\in S_n.
\]

Regard $q_n$ as an element of $S$.  Multiplicativity of the finite-variable
Gauss norm gives the uniform bound
\[
 \lVert q_n\rVert
 =\frac{\lVert a_n\rVert}{\lVert b_n\rVert}
 \leq\frac{\lVert a\rVert}{\lVert b\rVert}.
\]
Moreover,
\[
 a-bq_n=(a-a_n)-(b-b_n)q_n.
\]
The $c_0$ condition implies
$\lVert a-a_n\rVert,\lVert b-b_n\rVert\to0$.  Consequently
\[
 \left\lVert q-q_n\right\rVert_{\operatorname{Frac}(S)}
 \leq
 \frac{
 \max\{\lVert a-a_n\rVert,
       \lVert b-b_n\rVert\lVert q_n\rVert\}}
 {\lVert b\rVert}
 \longrightarrow0.
\]
Thus $(q_n)$ is Cauchy for the Gauss norm in $S$.  Its limit belongs to $S$
because $S$ is complete, and the isometric embedding
$S\hookrightarrow\operatorname{Frac}(S)$ identifies that limit with $q$.
Hence $q\in S$.
\end{proof}

\begin{definition}[The pinched algebra]\label{def:pinched-algebra}
Let $K$ be a perfectoid field of characteristic $p$, and choose
$\pi\in K$ with $0<|\pi|<1$.  Introduce the countable set of variables
\[
 \mathbf z=
 \{x_{j,\ast}:j\geq0\}\amalg
 \{x_{j,n}:j,n\geq0\}\amalg
 \{y_{j,n}:j,n\geq0\}.
\]
By the completed infinite-variable Tate algebra
$K\langle T,s,\mathbf z\rangle$ we mean the completion of the polynomial
algebra for the Gauss norm.  Thus its elements are null families of
coefficients indexed by the monomials, each of which has finite support.
Let $J$ be the closed ideal consisting of series all of whose monomials have
positive total $\mathbf z$-degree, and define
\begin{equation}\label{eq:pinched-ring}
 R=K\langle T,s\rangle+sJ
 \ \subset\ K\langle T,s,\mathbf z\rangle .
\end{equation}
The sum in \cref{eq:pinched-ring} is direct as a Banach space and is a closed
subring.  We endow $R$ with the restricted Gauss norm and use
$R^+=R^\circ$.

There is a useful monoid description.  Let $I$ index $\mathbf z$, and write a
monomial exponent as $(a,b,\alpha)\in\N^2\oplus\N^{(I)}$, where $a$ and $b$
are the exponents of $T$ and $s$.  Set
\begin{equation}\label{eq:pinched-monoid}
 M=\{(a,b,0):a,b\in\N\}
 \,\amalg\,
 \{(a,b,\alpha):\alpha\neq0,\ b\geq1\}.
\end{equation}
Then $R=K\langle M\rangle$, the $c_0$-completion of the monoid algebra
$K[M]$.
\end{definition}

\begin{lemma}[Normality of the pinching]\label{lem:pinched-normal}
For every complete extension $L/K$, the ring
$R\whotimes_KL\simeq L\langle M\rangle$ is a normal domain.
\end{lemma}

\begin{proof}
It suffices to work over an arbitrary complete coefficient field, which we
continue to denote by $K$.  Put
\[
 S=K\langle T,s,\mathbf z\rangle,\qquad
 D=K\langle T\rangle,\qquad
 B=S/sS=K\langle T,\mathbf z\rangle.
\]
The ring $D$ is relatively integrally closed in $B$.  Indeed, suppose that
$h\in B$ is integral over $D$, and let $h_0\in D$ be obtained by setting all
$\mathbf z$-variables equal to zero.  If $h-h_0\neq0$, choose a finite set
$F$ of $\mathbf z$-variables for which the corresponding projection
$h_F-h_0\in D\langle F\rangle$ is nonzero.  A monomial order on $\N^F$
shows that $h_F-h_0$ is transcendental over $\operatorname{Frac}(D)$: if
$\alpha>0$ is
its least monomial exponent, then in any nonzero polynomial in $h_F-h_0$
over $\operatorname{Frac}(D)$, the least nonvanishing power has the unique least exponent
equal to a multiple of $\alpha$.  This contradicts the projected integral
equation.  Hence $h=h_0\in D$.

The definition of the pinched ring is equivalently the pullback description
\begin{equation}\label{eq:pinched-pullback}
 R=\{f\in S:f\bmod sS\in D\}=K\langle T,s\rangle+sS.
\end{equation}
Moreover, $sS\subset R$ and $s\neq0$, so
$\operatorname{Frac}(R)=\operatorname{Frac}(S)$.  If
$q\in\operatorname{Frac}(R)$ is integral over $R$, then it is integral over
$S$.
By \Cref{lem:infinite-tate-normal}, we have $q\in S$.  Reducing a monic
equation for $q$ modulo $sS$ shows that $\bar q\in B$ is integral over $D$,
and therefore $\bar q\in D$.  The pullback description
\cref{eq:pinched-pullback} now gives $q\in R$.

The same proof applies after replacing $K$ by every complete extension
$L/K$, and the orthonormal monomial basis identifies the completed base
change with $L\langle M\rangle$.
\end{proof}

\begin{lemma}\label{lem:pinched-sousperfectoid}
The ring $R$ is geometrically normal in the sense of
\Cref{lem:pinched-normal}.  It is also sousperfectoid.  Consequently
$(R,R^\circ)$ is stably uniform and sheafy.
\end{lemma}

\begin{proof}
Geometric normality is \Cref{lem:pinched-normal}.  We record directly the
multiplicativity of the norm because it will also be used below.
The monoid $M$ embeds in the torsion-free abelian group
\[
 \mathbf Z^2\oplus\mathbf Z^{(I)}.
\]
Choose a translation-invariant total ordering of this group.  The usual
leading-term argument shows that the Gauss norm on $K[M]$ is multiplicative.
The same is true after completion.  Indeed, the coefficients of a nonzero
$c_0$-series attain their maximal norm at only finitely many exponents.  Taking
the least such exponent in each of two series gives a unique maximal
contribution to the corresponding coefficient of their product.  Hence
\[
 \lVert fg\rVert=\lVert f\rVert\lVert g\rVert
 \qquad(f,g\in K\langle M\rangle).
\]
In particular, $R$ is an integral domain and its Gauss norm is uniform.

It remains to prove the stronger stable uniformity assertion.  Let
\[
 M[1/p]=\bigcup_{r\geq0}p^{-r}M
 \subset
 \N[1/p]^2\oplus\N[1/p]^{(I)}.
\]
Explicitly, an exponent $(a,b,\alpha)$ belongs to $M[1/p]$ if and only if all
coordinates are nonnegative $p$-power rational numbers and
\begin{equation}\label{eq:pinched-perfection-condition}
 \alpha\neq0\quad\Longrightarrow\quad b>0.
\end{equation}
Put
\[
 \widetilde R=K\langle M[1/p]\rangle.
\]
The monoid $M[1/p]$ is $p$-divisible.  Since $K$ is perfect and has
characteristic $p$, taking coefficientwise $p$th roots and dividing exponents
by $p$ shows that Frobenius on $\widetilde R$ is surjective.  The preceding
leading-term argument gives a multiplicative Gauss norm on $\widetilde R$.
Thus $\widetilde R$ is uniform, and the characteristic-$p$ perfectoid criterion
shows that it is a perfectoid Tate ring
~\cite[Definition~6.1.1 and Remark~6.1.3]{ScholzeWeinstein}.

The inclusion $R\hookrightarrow\widetilde R$ has a continuous projection
\begin{equation}\label{eq:monoid-projection}
 P\colon\widetilde R\longrightarrow R,\qquad
 \sum_{q\in M[1/p]}a_qX^q\longmapsto\sum_{q\in M}a_qX^q.
\end{equation}
We check the essential point that $P$ is $R$-linear.  If
$q\in M[1/p]\setminus M$, then $q$ has a nonintegral coordinate.  In fact, an
integral point of $M[1/p]$ satisfies \cref{eq:pinched-perfection-condition}
and therefore belongs to $M$.  Adding any $m\in M$ does not change the
fractional parts of the coordinates of $q$, so $q+m\notin M$.  The complement
$M[1/p]\setminus M$ is therefore stable under addition by $M$, which proves
$R$-linearity of \cref{eq:monoid-projection}.  The projection has norm at most
one, so $R$ is a topological $R$-module direct summand of the perfectoid ring
$\widetilde R$.  This is precisely the sousperfectoid condition
~\cite[Definition~6.3.1]{ScholzeWeinstein}.

Rational localizations of a sousperfectoid ring are again sousperfectoid, and
sousperfectoid Tate--Huber pairs are stably uniform
~\cite[Propositions~6.3.3--6.3.4]{ScholzeWeinstein}.  Stably uniform pairs are
sheafy~\cite[Theorem~7]{BuzzardVerberkmoes}.  This proves the lemma.
\end{proof}

\begin{definition}[The cyclic action]\label{def:uniform-action}
On the ambient infinite-variable Tate
algebra, fix $K,T,s$ and set
\begin{align}
 \sigma(x_{j,\ast})&=x_{j,\ast}-\pi^j y_{j,0},
 \label{eq:uniform-action-star}\\
 \sigma(x_{j,n})&=x_{j,n}+\pi^j y_{j,n}-T y_{j,n+1},
 \label{eq:uniform-action-chain}\\
\sigma(y_{j,n})&=y_{j,n}.
\label{eq:uniform-action-y}
\end{align}
\end{definition}
The right-hand sides have Gauss norm at most one and are homogeneous of
$\mathbf z$-degree one, so the substitutions extend to a bounded
endomorphism and preserve the subring $R$.  Since the $y$-variables are fixed,
iteration adds the displayed $y$-linear increments.  Therefore
$\sigma^p=1$ in characteristic $p$.  Moreover
\[
 \sigma(x_{0,\ast})-x_{0,\ast}=-y_{0,0}\neq0
\]
in the ambient domain, so the restriction to $R$ has exact order $p$ as well:
for example,
\[
 \sigma(sx_{0,\ast})-sx_{0,\ast}=-sy_{0,0}\neq0.
\]
It preserves $R^\circ$ and defines a continuous $\Cp$-action on
$X=\Spa(R,R^\circ)$.

\section{Local invariants and the non-adic quotient}
\label{sec:uniform-quotient}

\begin{definition}[The distinguished point and bidiscs]
\label{def:uniform-bidiscs}
Let $x_0\in X$ be the $K$-rational point induced by the map
\[
 R\longrightarrow K,\qquad T,s,sJ\longmapsto0.
\]
It is fixed by $\Cp$.  For $\ell\geq0$, let
\begin{equation}\label{eq:uniform-neighborhoods}
 V_\ell=\{x\in X:
 |T(x)|\leq|\pi|^\ell,\ |s(x)|\leq|\pi|^{2\ell}\}.
\end{equation}
These are $\Cp$-stable rational subsets.  Put
\[
 D_\ell=K\left\langle\frac{T}{\pi^\ell},
                 \frac{s}{\pi^{2\ell}}\right\rangle
\]
and let $J_\ell$ denote the positive $\mathbf z$-degree ideal in
$D_\ell\langle\mathbf z\rangle$.
\end{definition}

\begin{lemma}[The local section rings]\label{lem:uniform-local-rings}
For every $\ell\geq0$, completed rational localization induces an isometric
identification
\begin{equation}\label{eq:uniform-local-ring}
 R_\ell:=\cO_X(V_\ell)=D_\ell+sJ_\ell.
\end{equation}
Under this identification, the decomposition by total $\mathbf z$-degree is
orthogonal and its degree projections are continuous.
\end{lemma}

\begin{proof}
Introduce degree-zero variables
\[
 U=\frac{T}{\pi^\ell},\qquad V=\frac{s}{\pi^{2\ell}}.
\]
The completed rational localization attached to $V_\ell$ is the completion of
\[
 R[U,V]/(\pi^\ell U-T,\pi^{2\ell}V-s)
\]
for the quotient Gauss norm.  Recall from \cref{eq:pinched-monoid} that the
monomials of $R$ are exactly
\[
 T^a s^b
 \quad\text{and}\quad
 T^a s^b\mathbf z^\alpha
 \quad(\alpha\neq0,\ b\geq1).
\]
After imposing the two displayed relations, these monomials become
\[
 \pi^{\ell a+2\ell b}U^aV^b
 \quad\text{and}\quad
\pi^{\ell a+2\ell b}U^aV^b\mathbf z^\alpha
\quad(\alpha\neq0,\ b\geq1).
\]
Eliminating $T$ and $s$ by the two relations gives a unique finite linear
combination of these normal monomials.  They form an orthogonal family for the
quotient norm.  Its $c_0$-completion
has arbitrary degree-zero part in $D_\ell=K\langle U,V\rangle$, whereas every
positive-$\mathbf z$-degree monomial is divisible by
$s=\pi^{2\ell}V$.  It is therefore precisely $D_\ell+sJ_\ell$, with the
restricted Gauss norm.  Conversely, finite sums of the displayed monomials
are dense in $D_\ell+sJ_\ell$, so this identification is onto and isometric.
The universal property of rational localization identifies this completed
ring with $\cO_X(V_\ell)$.  Orthogonality also proves the final assertion.
\end{proof}

The action preserves total $\mathbf z$-degree.  Hence the continuous degree
projections of \Cref{lem:uniform-local-rings} commute with the action.

\begin{lemma}\label{lem:uniform-cofinal}
The sequence $(V_\ell)_{\ell\geq0}$ is cofinal among the neighborhoods of
$x_0$ in $X$.
\end{lemma}

\begin{proof}
Every $f\in R$ has a unique decomposition
\[
 f=d(T,s)+a,\qquad d\in K\langle T,s\rangle,\quad a=sh,\quad h\in J.
\]
Choose $N\geq0$ such that $h'=\pi^Nh$ has Gauss norm at most one.  Although
$h'$ need not belong to $R$, the element $s(h')^2$ does belong to $R^\circ$.
For every $x\in X$ we therefore have
\begin{equation}\label{eq:square-root-bound}
 |\pi^Na(x)|^2=|s(x)|\,|s(h')^2(x)|\leq |s(x)|.
\end{equation}
The value group is totally ordered, so on $V_\ell$ this implies
\begin{equation}\label{eq:positive-degree-bound}
 |a(x)|\leq|\pi|^{\ell-N}.
\end{equation}

Now let
\[
 W=\{x:|f_i(x)|\leq|g(x)|\neq0,\ 1\leq i\leq r\}
\]
be a rational neighborhood of $x_0$.  Apply
\cref{eq:positive-degree-bound} to the positive-degree parts of $g$ and the
$f_i$.  The degree-zero parts differ from their values at $(T,s)=(0,0)$ by
$Td_1+sd_2$ for suitable $d_1,d_2\in K\langle T,s\rangle$.  After a common
rescaling by a power of $\pi$, their absolute values on $V_\ell$ are bounded
by $|\pi|^{\ell-N'}$ for some $N'$ independent of $\ell$.  Since
$g(x_0)\neq0$ and $|f_i(x_0)|\leq|g(x_0)|$, all these error terms are strictly
smaller than $|g(x_0)|$ for sufficiently large $\ell$.  The ultrametric
inequality then gives
\[
 |g(x)|=|g(x_0)|\neq0,\qquad
 |f_i(x)|\leq|g(x_0)|=|g(x)|
\]
on $V_\ell$.  Hence $V_\ell\subset W$.  Since rational subsets form a basis,
the claim follows.
\end{proof}

We compute only degree one, which is enough to detect the obstruction.  Its
monomial basis consists of
\[
 sx_{j,\ast},\quad sx_{j,n},\quad sy_{j,n}\qquad(j,n\geq0).
\]
For $0\leq j<\ell$, define
\begin{equation}\label{eq:uniform-u}
 u_{j,\ell}
 =x_{j,\ast}+\sum_{n\geq0}
 \left(\frac{T}{\pi^j}\right)^n x_{j,n}.
\end{equation}
The expression $u_{j,\ell}$ is notation in the ambient degree-one module;
the actual element of $R_\ell$ is $s u_{j,\ell}$.  It converges exactly when
$j<\ell$.

\begin{lemma}\label{lem:uniform-degree-one-invariants}
The degree-one part of the invariant ring $R_\ell^{\Cp}$ is
\begin{equation}\label{eq:uniform-degree-one-fixed}
 (R_\ell^{\Cp})_1
 =
 s\left(
 \bigoplus_{0\leq j<\ell}D_\ell u_{j,\ell}
 \ \oplus\
 c_0(\{(j,n):j,n\geq0\},D_\ell)\{y_{j,n}\}
 \right).
\end{equation}
\end{lemma}

\begin{proof}
On degree one, $\sigma-1$ is the $D_\ell$-linear map that kills every
$y_{j,n}$ and sends
\[
 x_{j,\ast}\longmapsto-\pi^jy_{j,0},\qquad
 x_{j,n}\longmapsto\pi^jy_{j,n}-Ty_{j,n+1}.
\]
For a vector
$a_\ast x_{j,\ast}+\sum_{n\geq0}a_nx_{j,n}$ in the $j$th $x$-block, the
kernel equations are
\[
 a_0=a_\ast,\qquad
 a_n=(T/\pi^j)a_{n-1}\quad(n\geq1).
\]
Thus the formal kernel vector is a multiple of $u_{j,\ell}$.  Since
\[
 \lVert T/\pi^j\rVert_{D_\ell}=|\pi|^{\ell-j},
\]
its coefficient family is null exactly for $j<\ell$.  The finitely many
surviving $x$-blocks and the full $c_0$-module spanned by the fixed
$y$-variables give \cref{eq:uniform-degree-one-fixed}.
\end{proof}

Let $q\colon X\to Y=X/\Cp$ be the quotient in the category of $v$-ringed
spaces, and put
\[
 W_\ell=q(V_\ell),\qquad B_\ell=\cO_Y(W_\ell)=R_\ell^{\Cp}.
\]
The equality for sections follows from the definition of the quotient sheaf.
The sets $W_\ell$ are open because a finite-group quotient map is open, and
they are cofinal among the neighborhoods of $y=q(x_0)$ by
\Cref{lem:uniform-cofinal}.

\begin{proposition}\label{prop:uniform-localization-fails}
For every $\ell\geq0$, the canonical map
\begin{equation}\label{eq:uniform-comparison}
 B_\ell\left\langle
 \frac{T}{\pi^{\ell+1}},\frac{s}{\pi^{2(\ell+1)}}
 \right\rangle
 \longrightarrow B_{\ell+1}
\end{equation}
is not surjective.  The element
\[
 s u_{\ell,\ell+1}\in B_{\ell+1}
\]
does not belong to its image.
\end{proposition}

\begin{proof}
Let $\lambda_\ell\colon R_{\ell+1}\to D_{\ell+1}$ be the continuous
coefficient functional extracting the coefficient of the monomial
$s x_{\ell,\ast}$.  This functional is $D_{\ell+1}$-linear on degree one and
vanishes on every other degree.  If $b\in B_\ell$, then
\Cref{lem:uniform-degree-one-invariants} shows that the degree-one part of its
image in $R_{\ell+1}$ belongs to
\[
 s\left(
 \bigoplus_{0\leq j<\ell}D_{\ell+1}u_{j,\ell+1}
 \ \oplus\
 c_0(\{(j,n):j,n\geq0\},D_{\ell+1})\{y_{j,n}\}
 \right).
\]
No vector in this subspace contains the monomial $s x_{\ell,\ast}$, and hence
$\lambda_\ell(b)=0$.

Let $C_\ell$ denote the source of \cref{eq:uniform-comparison}, and write
$\varphi_\ell\colon C_\ell\to B_{\ell+1}$ for the comparison map.  The image in
$C_\ell$ of the algebra
\[
 B_\ell[U,V],\qquad
 U=\frac{T}{\pi^{\ell+1}},\quad
 V=\frac{s}{\pi^{2(\ell+1)}},
\]
is dense by the definition of completed rational localization.  The images of
$U$ and $V$ have $\mathbf z$-degree zero.  Thus the
$D_{\ell+1}$-linearity just noted gives
\[
 \lambda_\ell\bigl(\varphi_\ell(bU^aV^c)\bigr)
 =U^aV^c\lambda_\ell(b)=0
 \qquad(b\in B_\ell,\ a,c\geq0).
\]
The continuous map $\lambda_\ell\circ\varphi_\ell$ therefore vanishes on a
dense subring of $C_\ell$, so it vanishes on all of $C_\ell$.  On the other
hand,
\[
 \lambda_\ell(su_{\ell,\ell+1})=1.
\]
The element $su_{\ell,\ell+1}$ is invariant by
\Cref{lem:uniform-degree-one-invariants}, so it belongs to $B_{\ell+1}$ and
cannot lie in the image.
\end{proof}

\begin{proof}[Proof of \Cref{thm:uniform-main}]
By \Cref{lem:pinched-sousperfectoid}, the ring $R$ is geometrically normal and
sousperfectoid, while $(R,R^\circ)$ is stably uniform and sheafy.  The formulas
\cref{eq:uniform-action-star,eq:uniform-action-chain,eq:uniform-action-y}
give the required action of exact order $p$.

Let $q\colon X\to Y=X/\Cp$ be the quotient and let $y=q(x_0)$.  By
\Cref{lem:uniform-cofinal}, the neighborhoods $W_\ell=q(V_\ell)$ are cofinal
at $y$.  Since $T$ and $s$ are invariant global sections, every $W_\ell$ is a
globally defined rational locus, and $W_{\ell+1}$ is intrinsically rational
inside $W_\ell$.  The corresponding completed rational-localization map is
\cref{eq:uniform-comparison}, which is not surjective by
\Cref{prop:uniform-localization-fails}.  Applying
\Cref{prop:persistent-obstruction}, we conclude that $y$ has no affinoid
neighborhood.  Hence $Y$ is not an adic space.
\end{proof}

\begin{proof}[Proof of \Cref{cor:no-unrestricted-quotient}]
The source in \Cref{thm:uniform-main} is affinoid and sousperfectoid, hence
stably uniform and sheafy, while its finite quotient is not adic.  If $R$
were perfectoid, Hansen's affinoid perfectoid finite-quotient theorem
~\cite[Theorem~1.4]{HansenQuotients} would make the quotient perfectoid and
therefore adic, a contradiction.  Thus $R$ is not perfectoid.  The source is
not topologically of finite type, so there is no conflict with Zavyalov's
positive finite-type theorem~\cite[Theorem~1.3.3]{ZavyalovQuotients}.
\end{proof}

\begin{remark}[Persistence under perfectoid base change]
If $L/K$ is an extension of perfectoid fields, then
$R\whotimes_KL\simeq L\langle M\rangle$.  The cyclic action, the cofinal
bidiscs, the new invariant sections, and the separating coefficient
functional have the same formulas over $L$.  Thus the base-changed quotient
is again non-adic.  This is a robustness property of the construction rather
than a separate obstruction.
\end{remark}

\subsection*{A purely inseparable byproduct}

For the particular missing section exhibited above, its $p$th power does
descend.  The next proposition concerns this witness only; it does not claim
that the full comparison map becomes surjective after perfection.

\begin{proposition}[An exhibited purely inseparable root]
\label{prop:uniform-pth-power-descends}
For \(\ell\geq0\), put
\[
 v_\ell=su_{\ell,\ell+1}\in B_{\ell+1}.
\]
The element \(v_\ell\) is not in the image of
\cref{eq:uniform-comparison}, but \(v_\ell^p\) is in its image.
\end{proposition}

\begin{proof}
The first assertion is \Cref{prop:uniform-localization-fails}.  Set
\(\tau=T/\pi^\ell\), and for \(N\geq0\) define
\[
 z_N=s\left(x_{\ell,\ast}+\sum_{n=0}^N\tau^n x_{\ell,n}\right)
 \in R_\ell.
\]
The formulas for the action telescope to
\[
 \delta_N:=(\sigma-1)z_N
 =-s\pi^\ell\tau^{N+1}y_{\ell,N+1}.
\]
This error is fixed by \(\sigma\), and
\(\sigma^a(z_N)=z_N+a\delta_N\) for \(a\in\mathbf F_p\).  The orbit norm
\begin{equation}\label{eq:truncated-orbit-norm}
 P_N=\prod_{a\in\mathbf F_p}\sigma^a(z_N)
 =\prod_{a\in\mathbf F_p}(z_N+a\delta_N)
 =z_N^p-z_N\delta_N^{p-1}
\end{equation}
therefore belongs to \(B_\ell\).

Write
\[
 a_N=z_{N+1}-z_N=s\tau^{N+1}x_{\ell,N+1}.
\]
Using characteristic \(p\) in \cref{eq:truncated-orbit-norm} gives
\[
 P_{N+1}-P_N
 =a_N^p-(z_N+a_N)\delta_{N+1}^{p-1}
   +z_N\delta_N^{p-1}.
\]
Each term is divisible in \(R_\ell\) by
\(\tau^{(N+1)(p-1)}\).  Thus
\begin{equation}\label{eq:orbit-norm-cauchy}
 P_{N+1}-P_N=\tau^{(N+1)(p-1)}H_N
\end{equation}
for some \(H_N\in R_\ell\) with
\(\sup_N\lVert H_N\rVert<\infty\).  Since the left side and \(\tau\) are
invariant and \(R_\ell\) is a domain, \(H_N\) is invariant, hence belongs to
\(B_\ell\).

In the source of \cref{eq:uniform-comparison}, one has
\(\lVert\tau\rVert\leq|\pi|<1\).  Equation
\cref{eq:orbit-norm-cauchy} shows that \((P_N)\) is Cauchy there.  In
\(B_{\ell+1}\), the sequence \(z_N\) converges to \(v_\ell\) and
\(\delta_N\) converges to zero.  Therefore the image of the limit of
\((P_N)\) is
\[
 \lim_{N\to\infty}
 \left(z_N^p-z_N\delta_N^{p-1}\right)=v_\ell^p,
\]
which proves the second assertion.
\end{proof}

\section*{Acknowledgements}
The authors thank Zida Wang for posing the question addressed in this paper.
The main results of this paper were obtained by Eureka and subsequently
verified by the authors. Eureka is a multi-agent system developed by
\mbox{JIUCHONG} at the University of Science and Technology of China for
mathematical research through human--AI interaction. The authors assume full
responsibility for the content of this paper. This work is supported by the
Fundamental and Interdisciplinary Disciplines Breakthrough Plan of the Ministry
of Education of China.

\end{document}